\documentclass[a4paper,12pt]{article}

\usepackage[left=2cm,right=2cm, top=2cm,bottom=3cm,bindingoffset=0cm]{geometry}

\usepackage{verbatim}
\usepackage{amsmath}
\usepackage{amsthm}
\usepackage{amssymb}
\usepackage{delarray}
\usepackage{cite}
\usepackage{hyperref}
\usepackage{mathrsfs}
\usepackage{tikz}
\usetikzlibrary{patterns}
\usepackage{caption}
\DeclareCaptionLabelSeparator{dot}{. }
\newcommand{\al}{\alpha}
\newcommand{\be}{\beta}
\newcommand{\ga}{\gamma}
\newcommand{\de}{\delta}
\newcommand{\la}{\lambda}
\newcommand{\om}{\omega}

\newcommand{\vv}{\varphi}
\newcommand{\iy}{\infty}

\theoremstyle{plain}

\numberwithin{equation}{section}

\newtheorem{thm}{Theorem}[section]
\newtheorem{lem}[thm]{Lemma}
\newtheorem{prop}[thm]{Proposition}

\theoremstyle{definition}

\newtheorem{alg}[thm]{Algorithm}
\newtheorem{ip}[thm]{Inverse Problem}

\theoremstyle{remark}

\newtheorem{remark}[thm]{Remark}

\DeclareMathOperator*{\Res}{Res}
\DeclareMathOperator{\rank}{rank}
\DeclareMathOperator{\diag}{diag}
\DeclareMathOperator{\Ran}{Ran}
\DeclareMathOperator{\Ker}{Ker}

 \allowdisplaybreaks

\begin{document}

\begin{center}
{\large\bf Direct and inverse problems for the matrix Sturm-Liouville operator with general self-adjoint boundary conditions}
\\[0.2cm]
{\bf Natalia P. Bondarenko} \\[0.2cm]
\end{center}

\vspace{0.5cm}

{\bf Abstract.} The matrix Sturm-Liouville operator on a finite interval with boundary conditions in the general self-adjoint form and with singular potential of class $W_2^{-1}$ is studied. This operator generalizes Sturm-Liouville operators on geometrical graphs. We investigate structural and asymptotical properties of the spectral data (eigenvalues and weight matrices) of this operator. Furthermore, we prove the uniqueness of recovering the operator from its spectral data, by using the method of spectral mappings.
  
\medskip

{\bf Keywords:} matrix Sturm-Liouville operator; singular potential; Sturm-Liouville operators on graphs; eigenvalue asymptotics; Riesz-basicity of eigenfunctions; inverse problem; uniqueness theorem. 

\medskip

{\bf AMS Mathematics Subject Classification (2010):} 34B09 34B24 34B45 34L10 34L20 34L40 34A55

\vspace{1cm}

\section{Introduction}

The paper is concerned with spectral theory of matrix Sturm-Liouville operators given by the differential expression $\ell Y = -Y'' + Q(x) Y$, where $Q(x) = [q_{jk}(x)]_{j, k = 1}^m$ is a matrix function called the {\it potential}. Such operators generalize scalar Sturm-Liouville operators, which have been studied fairly completely (see, e.g., the monographs \cite{Mar77, LS91, FY01}). In this paper, we consider the matrix Sturm-Liouville operator with boundary conditions in the general self-adjoint form. This operator causes interest, since it generalizes Sturm-Liouville operators on metric graphs. The latter operators are used for modeling various processes in graph-like structures in organic chemistry, mesoscopic physics, nanotechnology, microelectronics, and other applications (see \cite{Kuch04, PPP04, BCFK06, Now07} and references therein).

In order to provide the problem statement, denote by $\mathbb C^m$ and $\mathbb C^{m \times m}$ the spaces of complex $m$-vectors and $(m \times m)$-matrices, respectively. The notations $L_2((0, \pi); \mathbb C^m)$ and $L_2((0, \pi); \mathbb C^{m \times m})$ are used for the spaces of $m$-vector functions and $(m \times m)$-matrix functions, respectively, with elements from $L_2(0, \pi)$.

Consider the matrix Sturm-Liouville equation
\begin{equation} \label{eqvQ}
    -Y'' + Q(x) Y = \la Y, \quad x \in (0, \pi),
\end{equation}
where $Y = [y_j(x)]_{j = 1}^m$ is a vector function, $\la$ is the spectral parameter, $Q(x)$ is an $(m \times m)$ Hermitian matrix function with elements of class $W_2^{-1}(0, \pi)$, that is, $Q(x) = \sigma'(x)$, $\sigma \in L_2((0, \pi); \mathbb C^{m \times m})$, $\sigma(x) = (\sigma(x))^{\dagger}$, the symbol ``$\dagger$'' denotes the conjugate transpose. The derivatives of $L_2$-functions are understood in the sense of distributions.
Equation~\eqref{eqvQ} can be rewritten in the form
\begin{equation} \label{eqv}
    \ell Y := -(Y^{[1]})' - \sigma(x) Y^{[1]} - \sigma^2(x) Y = \la Y, \quad x \in (0, \pi),
\end{equation}
where $Y^{[1]}(x) := Y'(x) - \sigma(x) Y(x)$ is the {\it quasi-derivative}. Direct and inverse problem theory for the \textit{scalar} operators in the form~\eqref{eqv} has been developed in \cite{SS99, Sav01, SS01, Kor03, HM-sd, HM-trans, HM-2sp, HM-half, SS05, DM09, Hryn11, Mirz14, Gul19} and other studies.

Denote by $L$ the boundary value problem for equation~\eqref{eqv} with boundary conditions
\begin{gather}
   \label{bc1}
    V_1(Y) := T_1 (Y^{[1]}(0) - H_1 Y(0)) - T_1^{\perp} Y(0) = 0, \\ \label{bc2} V_2(Y) := T_2 (Y^{[1]}(\pi) - H_2 Y(\pi)) - T_2^{\perp} Y(\pi) = 0. 
\end{gather}
Here $T_j, T_j^{\perp}, H_j \in \mathbb C^{m \times m}$, $T_j$ are orthogonal projection matrices, that is, $T_j = T_j^{\dagger} = T_j^2$, $T_j^{\perp} = I - T_j$, $I$ is the unit matrix in $\mathbb C^{m \times m}$, $H_j = H_j^{\dagger} = T_j H_j T_j$, $j = 1, 2$. Under these assumptions, the problem $L$ is self-adjoint. We observe that, in the special cases $T_j = 0$ and $T_j = I$, the corresponding boundary condition turns into the Dirichlet and into the Robin boundary condition, respectively.

Denote by $\mathcal D(L)$ the space of $m$-vector functions $Y(x)$ such that the elements of $Y(x)$, $Y^{[1]}(x)$ are absolutely continuous on $[0, \pi]$, $\ell Y \in L_2((0, \pi); \mathbb C^{m})$, and $Y(x)$ satisfies \eqref{bc1}, \eqref{bc2}. The problem $L$ is related with the matrix Sturm-Liouville operator given by the differential expression $\ell Y$ on the domain $\mathcal D(L)$.

There is an extensive literature devoted to matrix Sturm-Liouville operators.
Asymptotic formulas and some other properties of spectral data have been obtained in \cite{Pap95, Car99, Vel07} for the second-order 
 matrix operators and in \cite{Pol19} for the fourth-order matrix operators. The uniqueness of recovering matrix Sturm-Liouville operators on a finite interval from various spectral characteristics has been proved in \cite{Car02, Mal05, Chab04, Yur06-uniq, Sh07, Xu19}. In \cite{Yur06}, a constructive algorithm was suggested for solving these inverse problems. Later on, the spectral data characterization for matrix Sturm-Liouville operators has been obtained in \cite{CK09, MT09, Bond11, Bond19}. The most of the mentioned studies concern operators with Dirichlet boundary conditions $Y(0) = Y(\pi) = 0$ and Robin boundary conditions $Y'(0) - H_1 Y(0) = Y'(\pi) + H_2 Y(\pi) = 0$. The boundary conditions \eqref{bc1}-\eqref{bc2} appear to be more difficult for investigation, because of more complicated behavior of the spectrum and structural properties. 
We know the only paper \cite{Xu19} on inverse problems for the matrix Sturm-Liouville operator with the general self-adjoint boundary conditions, but that paper is limited to uniqueness theorems (for square-integrable potential). It is also worth mentioning that direct and inverse scattering problems for the matrix Sturm-Liouville operator on the {\it half-line} with the boundary condition analogous to~\eqref{bc1} have been investigated in \cite{Har02, AW21}. Those studies generalize the approach of Agranovich and Marchenko \cite{AM63}, who considered the scattering problem with Dirichlet boundary condition. However, we find inverse problems for the matrix Sturm-Liouville operators on the half-line to be easier for investigation, than analogous problems on a finite interval, since the first problems have a bounded set of eigenvalues and there are no difficulties caused by asymptotic behavior of the spectrum.

The majority of studies on matrix Sturm-Liouville operators deal with the case of \textit{regular} potentials from the class $L_2$. Inverse problems for matrix Sturm-Liouville operators \textit{singular} potential of $W_2^{-1}$ on a finite interval were considered in the only paper \cite{MT09} by Mykytyuk and Trush. However, the authors of \cite{MT09} studied the operator in the form $-\left(\frac{d}{dx} + \tau \right)(\frac{d}{dx} - \tau) Y$, where $\tau$ is a square-integrable matrix function. This form differs from~\eqref{eqv} and can be conveniently reduced to a Dirac-type operator. 
The differential expression of Mykytyuk and Trush can be written in the form $-Y'' + Q(x) Y$ with Miura potential $Q = \tau' + \tau^2$ in the case of Dirichlet boundary conditions (but cannot for Neumann ones). 

We also mention that the reduction to Dirac-type operators analogous to \cite{MT09} was applied to matrix Sturm-Liouville operators with singular potentials on the half-line and on the line by Eckhardt at al \cite{Eckh14, Eckh15}. Some fundamental properties (maximal and minimal operator, deficiency indices, self-adjoint extensions, etc.) for matrix operators in our form \eqref{eqv} and in more general forms have been investigated by Weidmann \cite{Weid87} and by Mirzoev and Safonova \cite{MS11, MS16}.

The goal of this paper is two-fold. \textit{First}, we investigate properties of the spectral data (eigenvalues and weight matrices) of the problem $L$. Our spectral data generalize the classical spectral data $\{ \la_n, \al_n \}_{n \ge 1}$ of the scalar Sturm-Liouville operator $-y'' + q(x) y$, $q \in L_2(0, \pi)$, with boundary conditions $y(0) = y(\pi) = 0$, where $\{ \la_n \}_{n = 1}^{\iy}$ are the eigenvalues, $\{ y_n(x) \}_{n = 1}^{\iy}$ are the eigenfunctions normalized by the condition $y_n'(0) = 1$, and
$$
\al_n := \left( \int_0^{\pi} y_n^2(x) \, dx \right)^{-1}, \quad n \ge 1,
$$
(see \cite{Mar77, FY01}). The rigorous definition of the spectral data for the matrix Sturm-Liouville operator is provided in Section~2.
We study the structure and the asymptotic behavior of the spectral data, and also prove the completeness and the Riesz-basis property of a special sequence of vector functions constructed by the eigenvalues and the columns of the weight matrices. Such sequences play an important role in the spectral data characterization of the matrix Sturm-Liouville operators (see \cite{CK09, MT09, Bond19}). As a corollary, we show that the sequence of vector eigenfunctions of the problem $L$ is a Riesz basis. In particular, all these results are valid for the Sturm-Liouville operators on graphs with singular potentials and with rationally-dependent edge lengths.

\textit{Second}, we study the inverse problem that consists in recovering the potential and the boundary condition coefficients of the problem $L$ from the spectral data. We prove the corresponding uniqueness theorem, by developing the ideas of the method of spectral mappings \cite{FY01, Yur06, Bond-preprint}. We also discuss reconstruction of the potential from the Weyl matrix, consider the case of the square-integrable potential $Q(x)$, and compare our theorems with the known results. In the sequel study~\cite{Bond-sequel}, our approach gives a constructive solution of the inverse problem and the characterization of the spectral data.

The paper is organized as follows. In Section~2, we introduce the notions of Weyl matrix and weight matrices and study structural properties of the spectral characteristics. 
In Section~3, the spectral data of the problem $L$ with $\sigma = H_1 = H_2 = 0$ is explicitly found.
In Section~4, asymptotic formulas are derived for the eigenvalues, the weight matrices, and for solutions of equation~\eqref{eqv}. In Section~5, we prove the completeness and the Riesz-basis property of a special sequence of vector functions related to the problem $L$.
In Section~6, the inverse problems are studied and the corresponding uniqueness theorems are obtained. In the Appendix, we describe the reduction of the Sturm-Liouville eigenvalue problems on graphs to the matrix form~\eqref{eqv}-\eqref{bc2}.  

Throughout the paper, we use the notations:

\begin{itemize}
    \item $\rho := \sqrt{\la}$, $\arg \rho \in [-\tfrac{\pi}{2}, \tfrac{\pi}{2})$ (unless stated otherwise), $\tau := \mbox{Im}\, \rho$.
    \item We use the following vector norm in $\mathbb C^m$:
$$
    \| a \| = \left( \sum_{j = 1}^m |a_j|^2 \right)^{1/2}, \quad a = [a_j]_{j = 1}^m,    
$$
and the corresponding matrix norm  $\| A \| = s_{max}(A)$, where $s_{max}(A)$
is the maximal singular value of $A$.
    \item The scalar product in the Hilbert space
    $L_2((0, \pi); \mathbb C^m)$ is defined as follows:
    \begin{gather} \label{defscal}
    (Y, Z) = \int_0^{\pi} (Y(x))^{\dagger} Z(x) \, dx = \sum_{j = 1}^m \int_0^{\pi} \overline{y_j(x)} z_j(x) \, dx, \\ \nonumber Y = [y_j(x)]_{j = 1}^m, Z = [z_j(x)]_{j = 1}^m \in L_2((0, \pi); \mathbb C^m).
    \end{gather}
\end{itemize}

\section{Structural properties}

In this section, we introduce the notions of Weyl matrix and weight matrices and study the structure of the spectral characteristics of the problem $L$.

\begin{lem} \label{lem:sym}
For any functions $Y, Z \in \mathcal D(L)$, the relation $(\ell Y, Z) = (Y, \ell Z)$ holds. Thus, the operator induced by the differential expression $\ell$ and the boundary conditions~\eqref{bc1},\eqref{bc2} is symmetric, its eigenvalues are real, and vector eigenfunctions, corresponding to distinct eigenvalues, are orthogonal in $L_2((0, \pi); \mathbb C^m)$. 
\end{lem}

\begin{proof}
Consider arbitrary vector functions $Y, Z \in \mathcal D(L)$. Using~\eqref{eqv}, \eqref{defscal} and integration by parts, we obtain
\begin{align} \nonumber
(\ell Y, Z) & = -\int_0^{\pi} ((Y^{[1]})')^{\dagger} Z \, dx - \int_0^{\pi} (Y^{[1]})^{\dagger} \sigma Z \, dx - \int_0^{\pi} Y^{\dagger}\sigma^2 Z \, dx \\ \label{sm1}
& = -(Y^{[1]})^{\dagger} Z \Big|_0^{\pi} + \int_0^{\pi} (Y^{[1]})^{\dagger} Z^{[1]}\, dx - \int_0^{\pi} Y^{\dagger}\sigma^2 Z \, dx 
= (Y^{\dagger} Z^{[1]} - (Y^{[1]})^{\dagger} Z) \Big|_0^{\pi} + (Y, \ell Z)
\end{align}
The boundary condition~\eqref{bc1} yields
$$
T_1^{\perp} Y(0) = 0, \quad T_1 Y^{[1]}(0) = H_1 Y(0).
$$
The similar relations also hold for $Z$. Consequently, we have
\begin{align*}
(Y(0))^{\dagger} Z^{[1]}(0) - (Y^{[1]}(0))^{\dagger} Z(0) & = 
(Y(0))^{\dagger} (T_1 + T_1^{\perp}) Z^{[1]}(0) - (Y^{[1]}(0))^{\dagger} (T_1 + T_1^{\perp}) Z(0) \\ & = (Y(0))^{\dagger} H_1 Z(0) - (Y(0))^{\dagger} H_1 Z(0) = 0.
\end{align*}
Analogously, the substitution at $x = \pi$ in \eqref{sm1} also vanishes, so relation~\eqref{sm1} yields the claim.
\end{proof}

Let $\vv(x, \la)$ be the matrix solution of equation~\eqref{eqv} satisfying the initial conditions $\vv(0, \la) = T_1$, $\vv^{[1]}(0, \la) = T_1^{\perp} + H_1$. Clearly, $V_1(\vv) = 0$. For each fixed $x \in [0, \pi]$, the matrix functions $\vv(x, \la)$ and $\vv^{[1]}(x, \la)$ are entire in the $\la$-plane.

\begin{lem} \label{lem:mult}
The eigenvalues of the boundary value problem $L$ coincide with the zeros of the characteristic function $\Delta(\la) := \det(V_2(\vv(x, \la)))$ counting with their multiplicities. This means that, for every eigenvalue, the multiplicity of the zero of the analytical function $\Delta(\la)$ equals the number of linearly independent vector eigenfunctions corresponding to this eigenvalue. 
\end{lem}

Lemma~\ref{lem:mult} follows from the general theory of linear differential operators provided in the book of Naimark \cite{Nai68} (see Chapter I, \S 2, p.3 and Chapter III, \S 1, p.7). In addition, one can prove Lemma~\ref{lem:mult} similarly to \cite[Lemma~3]{Bond11}, \cite[Lemma~5]{Bond19} or \cite[Proposition~3.1]{Xu19}. 

Below, speaking about the roots of an analytic function or about the eigenvalues of some problem, we always count each value the number of times equal to its multiplicity.

\begin{thm} \label{thm:asymptla}
The spectrum of $L$ is a countable set of real eigenvalues $\{ \la_{nk} \}_{(n, k) \in J}$, numbered in non-decreasing order: $\la_{n_1 k_1} \le \la_{n_2 k_2}$ if $(n_1, k_1) < (n_2, k_2)$. The following asymptotic relation holds:
\begin{equation} \label{asymptla}
    \rho_{nk} := \sqrt{\la_{nk}} = n + r_k + \varkappa_{nk}, \quad (n, k) \in J, \quad \{ \varkappa_{nk} \} \in l_2,
\end{equation}
where
\begin{gather} \nonumber
J := \{ (n, k) \colon n \in \mathbb N, \, k = \overline{1, m} \} \cup \{ (0, k) \colon k = \overline{p^{\perp} + 1, m} \}, \\ \label{defpperp} p^{\perp} := \dim(\Ker T_1 \cap \Ker T_2),
\end{gather}
$\{ r_k \}_{k = 1}^m$ are the zeros of the function $w^0(\rho) := \det(W^0(\rho))$ on $[0, 1)$,
\begin{equation} \label{defW0}
W^0(\rho) := (T_2 T_1 + T_2^{\perp} T_1^{\perp}) \sin \rho \pi + (T_2^{\perp} T_1 - T_2 T_1^{\perp}) \cos \rho \pi.
\end{equation}
\end{thm}

Theorem~\ref{thm:asymptla} is proved is Section 4.
Now we proceed to define the weight matrices. Consider the boundary condition
\begin{equation} \label{bc3}
V_1^{\perp}(Y) := T_1 Y(0) + T_1^{\perp} Y^{[1]}(0) = 0.
\end{equation}

Let $\psi(x, \la)$ be the matrix solution of equation~\eqref{eqv} satisfying the initial conditions $\psi(0, \la) = -T_1^{\perp}$, $\psi^{[1]}(0, \la) = T_1$. One can easily check that $V_1^{\perp}(\psi) = 0$, $V_1(\psi) = I$, $V_1^{\perp}(\vv) = I$.

The {\it Weyl solution} of $L$ is the matrix solution $\Phi(x, \la)$ of equation~\eqref{eqv} satisfying the boundary conditions $V_1(\Phi) = I$, $V_2(\Phi) = 0$. The matrix function $M(\la) := V_1^{\perp}(\Phi)$ is called the {\it Weyl matrix} of the problem $L$. The notion of Weyl matrix generalizes Weyl function, which is a natural spectral characteristic in inverse problem theory (see \cite{Mar77, FY01}).

One can easily derive the relations
\begin{gather} \label{relPhi}
    \Phi(x, \la) = \psi(x, \la) + \vv(x, \la) M(\la), \\ \label{relM}
    M(\la) = -(V_2(\vv))^{-1} V_2(\psi), \\ \label{PhiPsi}
    \Phi(x, \la) = \Psi(x, \la) (V_1(\Psi))^{-1}, 
\end{gather}
where $\Psi(x, \la)$ is the solution of equation~\eqref{eqv} under the initial conditions $\Psi(\pi, \la) = T_2$, $\Psi^{[1]}(\pi, \la) = T_2^{\perp} + H_2$.
It follows from~\eqref{relPhi} and~\eqref{relM} that the matrix functions $M(\la)$ and $\Phi(x, \la)$ for each fixed $x \in [0, \pi]$ are meromorphic in the $\la$-plane with the poles at the eigenvalues of $L$.

\begin{lem} \label{lem:ranks}
All the poles of $M(\la)$ are simple, and the ranks of the residue-matrices coincide with the multiplicities of the corresponding eigenvalues of $L$.
\end{lem}

Lemma~\ref{lem:ranks} can be proved similarly to~\cite[Lemma~4]{Bond11}. Denote
$$
\al_{nk} := \Res_{\la = \la_{nk}} M(\la), \quad (n, k) \in J.
$$

The matrices $\{ \al_{nk} \}_{(n, k) \in J}$ are called the {\it weight matrices} and the data $\{ \la_{nk}, \al_{nk} \}_{(n, k) \in J}$ are called the \textit{spectral data} of $L$. 

Without loss of generality, below we assume that $H_1 = 0$. One can achieve this condition, applying the following transform: 
$$
\sigma(x) := \sigma(x) + H_1, \quad H_1 := 0, \quad H_2 := H_2 - T_2 H_1 T_2.
$$
Obviously, this transform does not change the spectral data $\{ \la_{nk}, \al_{nk} \}_{(n, k) \in J}$.

Now proceed to study properties of the weight matrices. Note that, if $Y$ and $Z$ satisfy equation~\eqref{eqv}, then the matrix Wronskian $\langle Y^{\dagger}, Z \rangle = (Y(x, \overline{\la}))^{\dagger} Z^{[1]}(x,\la) - (Y^{[1]}(x, \overline{\la}))^{\dagger} Z(x, \la)$ does not depend on $x$. 
Therefore, we obtain
$$
\langle (\Phi(x, \overline{\la}))^{\dagger}, \Phi(x, \la) \rangle = \langle (\Phi(x, \overline{\la}))^{\dagger}, \Phi(x, \la) \rangle_{|x = 0} = \langle (\Phi(x, \overline{\la}))^{\dagger}, \Phi(x, \la) \rangle_{|x = \pi}.
$$

It follows from~\eqref{relPhi} that
$$
\Phi(0, \la) = -T_1^{\perp} + T_1 M(\la), \quad 
\Phi^{[1]}(0, \la) = T_1 + T_1^{\perp} M(\la).
$$
Consequently, $\langle \Phi^{\dagger}, \Phi \rangle_{|x = 0} = M^{\dagger}(\overline{\la}) - M(\la)$. Since $\Phi(x, \la)$ satisfies \eqref{bc2}, it follows that $\langle \Phi^{\dagger}, \Phi \rangle_{|x = \pi} = 0$. We conclude that $M(\la) \equiv (M(\overline{\la}))^{\dagger}$. Hence, $\al_{nk} = \al_{nk}^{\dagger}$, $(n, k) \in J$.

\begin{lem} \label{lem:relal}
The following relations hold for $(n, k), (l, j) \in J$:
\begin{gather} \label{Val}
    V_2(\vv(x, \la_{nk})) \al_{nk} = 0, \\ \label{sym1}
    \al_{nk} \int_0^{\pi} (\vv(x, \la_{nk}))^{\dagger} \vv(x, \la_{lj}) \, dx \, \al_{lj} = \begin{cases}
                    \al_{nk}, \quad \la_{nk} = \la_{lj}, \\
                    0, \quad \la_{nk} \ne \la_{lj}.
                 \end{cases}
\end{gather}
\end{lem}

Lemma~\ref{lem:relal} can be proved analogously to \cite[Lemma~2.2]{Bond20}.

\section{Zero case}

In this section, the problem \eqref{eqv}-\eqref{bc2} is considered in the case $\sigma = 0$ in $L_2((0, \pi); \mathbb C^{m \times m})$, $H_1 = H_2 = 0$. We agree to use the superscript $0$ for objects corresponding to this special case.

One can easily show that
\begin{gather} \label{phi0}
\vv^0(x, \la) = \cos \rho x \, T_1 +  \frac{\sin \rho x}{\rho} T_1^{\perp}, \quad
\psi^0(x, \la) = \frac{\sin \rho x}{\rho} T_1 - \cos \rho x \, T_1^{\perp}, \\ \label{V20}
V_2^0(\vv^0) = - (\rho T_2 + T_2^{\perp}) W^0(\rho) (T_1 + \rho^{-1} T_1^{\perp}), \\ \label{V20psi}
V_2^0(\psi^0) = (\rho T_2 + T_2^{\perp}) U^0(\rho) (\rho^{-1} T_1 + T_1^{\perp}), \\
\label{defU0}
U^0(\rho) := (T_2 T_1 + T_2^{\perp} T_1^{\perp}) \cos \rho \pi + (T_2 T_1^{\perp} - T_2^{\perp} T_1) \sin \rho \pi,
\end{gather}
where $W^0(\rho)$ is defined in \eqref{defW0}. 

Let us find the eigenvalues of $L^0$.
In view of Lemma~\ref{lem:mult} and \eqref{V20}, the square roots of nonzero eigenvalues of the problem $L^0$ coincide with the zeros of the function $w^0(\rho) = \det(W^0(\rho))$. This function can be represented in the form
\begin{equation} \label{poly}
w^0(\rho) = (\sin \rho \pi)^{d_1} (\cos \rho \pi)^{d_2} P_{d_3}(\cos^2 \rho \pi),
\end{equation}
where $d_j$, $j = \overline{1, 3}$, are non-negative integers, $P_{d_3}(x)$ is a polynomial of degree $d_3$, $P(0) \ne 0$, $P(1) \ne 0$, and $d_1 + d_2 + 2 d_3 = m$. Consequently, the function $w^0(\rho)$ is either periodic or antiperiodic with period $1$. Note that the polynomial $P_{d_3}(x)$ has exactly $d_3$ roots on $(0, 1)$. Otherwise the function $w^0(\rho)$ has non-real roots, and this contradicts to the self-adjointness of the problem $L^0$. Therefore, the function $w^0(\rho)$ has exactly $m$ roots on $[0, 1)$. Denote them by $\{ r_k \}_{k = 1}^m$ in non-decreasing order. It follows from~\eqref{poly} that $w^0(\rho) = \pm w^0(1 - \rho)$, so for any $r_k \ne 0$ there exists $r_s = 1 - r_k$. The set of all zeros of $w^0(\rho)$ has the form 
\begin{equation} \label{rho0}
\rho_{nk}^0 = n + r_k, \quad n \in \mathbb Z, \quad k = \overline{1, m}.
\end{equation}
Consequently, the non-zero eigenvalues of $L^0$ have the form
$$
\la_{nk}^0 = (\rho_{nk}^0)^2, \quad n \ge 0, \quad k = \overline{1, m}, \quad \rho_{nk}^0 \ne 0.
$$

Let us separately study the case $\la = 0$.

\begin{lem} \label{lem:zero1}
The multiplicity of the eigenvalue $\la = 0$ of the problem $L^0$ equals $p := \dim (\Ran T_1 \cap \Ran T_2)$.
\end{lem}

\begin{proof}
The eigenfunctions corresponding to the eigenvalue $\la = 0$ of the problem $L^0$ have the form $\vv^0(x, 0) c$, where vectors $c \in \mathbb C^m$ are such that 
\begin{equation} \label{sm2}
    V_2^0(\vv^0(x, 0)) c = 0.
\end{equation} 
Using~\eqref{defW0} and~\eqref{V20}, we get 
$$
V_2^0(\vv(x, 0)) = T_2 T_1^{\perp} - T_2^{\perp} T_1 - \pi T_1^{\perp} T_2^{\perp}.
$$
Clearly, for any $c \in \Ran T_1 \cap \Ran T_2$, relation~\eqref{sm2} holds. Let us show that there are no other such $c$. Suppose that $c = c_1 + c_2 + c_3$, where $c_1$ and $c_2$ belong to the subspaces $\Ran T_1 \cap \Ran T_2$ and $\Ker T_1 \cap \Ker T_2$, respectively, and $c_3$ is orthogonal to both these subspaces. Then
\begin{equation} \label{sm3}
V_2^0(\vv^0(x, 0)) c = (T_2 T_1^{\perp} - T_2^{\perp} T_1) c_3 - \pi c_2.  
\end{equation}
If \eqref{sm2} holds, then
$$
(c_2, c_2) = \tfrac{1}{\pi}((T_2 T_1^{\perp} - T_2^{\perp} T_1) c_3, c_2) =
\tfrac{1}{\pi}(c_3, (T_1^{\perp} T_2 - T_1 T_2^{\perp}) c_2) = 0.
$$
Hence, $c_2 = 0$. It follows from~\eqref{sm2} and~\eqref{sm3} that $c_3 = 0$. Consequently, the number of linearly independent vectors $c$ satisfying~\eqref{sm2} equals $p$. This yields the claim.
\end{proof}

\begin{lem} \label{lem:zero2}
The multiplicity of the zero $\rho = 0$ of the function $w^0(\rho)$ equals $(p + p^{\perp})$, where $p^{\perp} := \dim(\Ker T_1 \cap \Ker T_2)$.
\end{lem}

\begin{proof}
Relying on Lemma~\ref{lem:mult}, one can show that the desired multiplicity equals $\dim(\Ker W^0(0))$. Let $c \in \Ker W^0(0)$. Represent $c$ in the form $c = c_1 + c_2$, $c_1 \in \Ran T_1$, $c_2 \in \Ker T_1$. Using~\eqref{defW0}, we get
$$
W^0(0) c = (T_2 T_1^{\perp} - T_2^{\perp} T_1) (c_1 + c_2) = T_2 c_2 - T_2^{\perp} c_1 = 0.
$$
This implies $T_2 c_2 = T_2^{\perp} c_1 = 0$, i.e., $c_1 \in \Ran T_2$, $c_2 \in \Ker T_2$. Therefore,
$$
    \Ker W^0(0) = (\Ran T_1 \cap \Ran T_2) \cup (\Ker T_1 \cap \Ker T_2).
$$
Thus, $\dim (\Ker W^0(0)) = p + p^{\perp}$.
\end{proof}

Lemmas~\ref{lem:zero1} and~\ref{lem:zero2} together with the arguments above them yield the assertion of Theorem~\ref{thm:asymptla} for $L^0$ with $\varkappa_{nk}^0 = 0$, $(n, k) \in J$.

Using~\eqref{defW0} and \eqref{poly}, we obtain the important estimate:
\begin{equation} \label{W0below}
\| (W^0(\rho))^{-1} \| \le C_{\de} \exp(-|\tau|\pi), \quad \rho \in
G_{\de}, 
\end{equation}
where 
\begin{equation} \label{Gde}
G_{\de} := \{ \rho \in \mathbb C \colon |\rho - \rho_{nk}^0| \ge \de, \, n \in \mathbb Z, \, k = \overline{1, m} \}, \quad \de > 0,
\end{equation}
and the constant $C_{\de}$ depends on $\de$.

Let us proceed to find the weight matrices $\{ \al_{nk}^0 \}$. Substituting~\eqref{V20} and~\eqref{V20psi} into \eqref{relM}, we obtain
\begin{equation} \label{ME0}
M^0(\la) = (T_1 + \rho T_1^{\perp}) E^0(\rho) (\rho^{-1} T_1 + T_1^{\perp}), \quad E^0(\rho) := (W^0(\rho))^{-1} U^0(\rho).
\end{equation}

It follows from~\eqref{defW0} and~\eqref{defU0} that the matrix function $E^0(\rho)$ is $1$-periodic and meromorphic in $\rho$ with the poles at $\rho = \rho_{nk}^0$. For $\rho_{nk}^0 \ne 0$, we have
\begin{gather} \label{weight1}
\al_{nk}^0 = \Res_{\la = \la_{nk}^0} M^0(\la) = 
2 \Res_{\rho = \rho_{nk}^0} (T_1 + \rho T_1^{\perp}) E^0(\rho) (T_1 + \rho T_1^{\perp}) = 
\frac{2}{\pi}(T_1 + \rho_{nk}^0 T_1^{\perp}) A_k (T_1 + \rho_{nk}^0 T_1^{\perp}), \\ \label{defAk}
A_k := \pi \Res_{\rho = r_k} E^0(\rho), \quad k = \overline{1, m}.
\end{gather}
On the other hand,
\begin{equation} \label{weight2}
\al_{nk}^0 = 2 \Res_{\rho = -\rho_{nk}^0} (T_1 + \rho T_1^{\perp}) E^0(\rho) (T_1 + \rho T_1^{\perp}) = \frac{2}{\pi} (T_1 - \rho_{nk}^0 T_1^{\perp}) A_s (T_1 - \rho_{nk}^0 T_1^{\perp}),
\end{equation}
where $r_k + r_s = 1$ or $r_k = r_s = 0$, $\rho_{nk}^0 \ne 0$.
Comparing~\eqref{weight1} and~\eqref{weight2}, we get
\begin{equation} \label{Aks}
    A_k = (T_1 - T_1^{\perp}) A_s (T_1 - T_1^{\perp}), \quad r_k + r_s = 1 \:\: \text{or} \:\: r_k = r_s = 0.
\end{equation}
The case $\rho_{nk}^0 = 0$ is slightly different:
\begin{equation} \label{weight0}
\al_{nk}^0 = \tfrac{1}{\pi} T_1 A_k T_1, \quad \rho_{nk}^0 = 0.
\end{equation}

By Lemma~\ref{lem:ranks}, we have
\begin{equation} \label{rankAk}
\rank (A_k) = \# \{ s = \overline{1, r} \colon r_s = r_k \}.
\end{equation}
Lemmas~\ref{lem:ranks}, \ref{lem:zero1}, \ref{lem:zero2} together with relations~\eqref{Aks}, \eqref{weight0} imply
\begin{equation} \label{A1}
    A_1 = T_1 A_1 T_1 + T_1^{\perp} A_1 T_1^{\perp}, \quad 
    \rank(T_1 A_1 T_1) = p, \quad \rank (T_1^{\perp} A_1 T_1^{\perp}) = p^{\perp}, \quad \text{if} \:\: r_1 = 0.
\end{equation}
Using \eqref{sym1} for $\al_{nk}^0$, \eqref{phi0}, \eqref{rho0}, \eqref{weight1},
and the relation $\al_{nk}^0 = (\al_{nk}^0)^{\dagger}$, we obtain
\begin{equation} \label{propAk}
A_k = A_k^{\dagger} = A_k^2, \quad A_k A_s = 0, \quad r_k \ne r_s,
\quad k,s = \overline{1,m}.
\end{equation}
Hence, $\{ A_k \}_{k \in \mathcal J}$ are matrices of mutually orthogonal projectors, $\mathcal J := \{1 \} \cup \{ k = \overline{2, m} \colon r_{k-1} \ne r_k \}$. In view of~\eqref{rankAk}, we have
\begin{equation} \label{sumAk}
    \sum_{k \in \mathcal J} A_k = I.
\end{equation}
Thus, we have explicitly described the weight matrices of the problem $L^0$.

\section{Asymptotics}

The goal of this section is to derive asymptotic formulas for the eigenvalues and for the weight matrices of the problem~$L$. We start with asymptotic formulas for solutions of equation~\eqref{eqv}.

Let $S(x, \la)$ and $C(x, \la)$ be the matrix solutions of equation~\eqref{eqv} satisfying the initial conditions
$$
S(0, \la) = C^{[1]}(0, \la) = 0, \quad S^{[1]}(0, \la) = C(0, \la) = I.
$$

The following theorem represents $S(x, \la)$, $C(x, \la)$ and their quasi-derivatives in terms of transformation operators. Such operators were introduced by Marchenko \cite{Mar77} for the classical case of regular potentials and play an important role in spectral theory.

\begin{thm} \label{thm:trans}
The following relations hold
\begin{align*}
    S(x, \la) &= \frac{\sin \rho x}{\rho} + \int_0^x \mathscr K_1(x, t) \frac{\sin \rho t}{\rho} \, dt, \\
    S^{[1]}(x, \la) & = \cos \rho x + \int_0^x \mathscr K_2(x, t) \cos \rho t \, dt, \\
    C(x, \la) & = \cos \rho x + \int_0^x \mathscr K_3(x, t) \cos \rho t \, dt, \\
    C^{[1]}(x, \la) & = - \rho \sin \rho x + \rho \int_0^x \mathscr K_4(x, t) \sin \rho t \, dt + \mathscr C(x),
\end{align*}
where the matrix functions $\mathscr K_j$, $j = \overline{1, 4}$, are square integrable in the region $\{ (x, t) \colon 0 < t < x < \pi \}$ and the matrix function $\mathscr C$ is continuous on $[0, \pi]$. Moreover, for each fixed $x \in (0, \pi]$, the functions $\mathscr K_j(x, .)$, $j = \overline{1, 4}$, belong to $L_2((0, x); \mathbb C^{m \times m})$ and the corresponding $L_2$-norms $\| \mathscr K_j(x, .) \|_{L_2((0, x); \mathbb C^{m \times m})}$ are uniformly bounded with respect to $x \in (0, \pi]$. Analogously, for each fixed $t \in [0, \pi)$, the functions $\mathscr K_j(., t)$, $j = \overline{1, 4}$, belong to $L_2((t, \pi); \mathbb C^{m \times m})$ and the corresponding $L_2$-norms are uniformly bounded with respect to $t \in [0, \pi)$.
\end{thm}

Theorem~\ref{thm:trans} is proved by reduction of the initial value problems for equation~\eqref{eqv} to systems of Volterra integral equations and then applying the method of successive approximations (see \cite{Bond-preprint}).

Since 
\begin{equation} \label{phiCS}
\vv(x, \la) = C(x, \la) T_1 + S(x, \la) T_1^{\perp},
\end{equation}
it follows that
\begin{multline*}
V_2(\vv(x, \la)) = T_2 (C^{[1]}(\pi, \la) - H_2 C(\pi, \la)) T_1 + T_2 (S^{[1]}(\pi, \la) - H_2 S(\pi, \la)) T_1^{\perp} \\ - T_2^{\perp} C(\pi, \la) T_1 - T_2^{\perp} S(\pi, \la) T_1^{\perp}.
\end{multline*}
Using Theorem~\ref{thm:trans}, we get
\begin{equation} \label{V2}
V_2(\vv) = - (\rho T_2 + T_2^{\perp}) W(\rho) (T_1 + \rho^{-1} T_1^{\perp}), \quad W(\rho) = W^0(\rho) + K(\rho).
\end{equation}
Here and below, the same notation $K(\rho)$ is used for various matrix functions of the form
\begin{equation} \label{defK}
K(\rho) = \int_{-\pi}^{\pi} \mathscr P(t) \exp(i \rho t) \, dt + \frac{\mathscr Q}{\rho}, \quad \mathscr P \in L_2((-\pi, \pi); \mathbb C^{m \times m}), \quad \mathscr Q \in \mathbb C^{m \times m}.
\end{equation}

In order to prove Theorem~\ref{thm:asymptla}, we need the following matrix version of Rouche's Theorem (see \cite[Lemma~2.2]{CK09}).

\begin{prop} \label{prop:Rouche}
Let $F(\rho)$ and $G(\rho)$ be matrix functions analytic in the disk $|\rho - a| \le r$ and satisfying the condition $\| G(\rho) F^{-1}(\rho) \| < 1$ on the boundary $|\rho - a| = r$. Then the scalar functions $\det(F)$ and $\det(F + G)$ have the same number of zeros inside the circle $|\rho - a| < r$.
\end{prop}

\begin{proof}[Proof of Theorem~\ref{thm:asymptla}]
{\bf Step 1.} Consider the functions $F(\rho) = \rho W^0(\rho)$ and $G(\rho) = \rho(W^0(\rho) - W(\rho))$ entire in the $\rho$-plane. Using~\eqref{V2} and~\eqref{defK}, we obtain
\begin{equation} \label{difW}
W^0(\rho) - W(\rho) = o(\exp(|\tau|\pi)), \quad |\rho| \to \iy.
\end{equation}
The estimates~\eqref{W0below} and \eqref{difW} yield $\| G(\rho) F^{-1}(\rho) \| < 1$ for sufficiently large $|\rho|$, sufficiently small $\de$, and $\rho \in G_{\de}$ ($G_{\de}$ is defined in \eqref{Gde}). Applying Proposition~\ref{prop:Rouche} to the contours $\{ |\rho| = R \} \subset G_{\de}$ and $\{ |\rho - \rho_{nk}^0| = \de \}$ with sufficiently large $R > 0$ and $|n|$ and sufficiently small $\de > 0$, we conclude that the functions $\rho^m \det(W^0(\rho))$ and $\rho^m \det(W(\rho))$ have the same number of zeros inside these contours. Hence, the function $\rho^m \det(W(\rho))$ has a countable set of zeros $\{ \theta_k \}_{k = 1}^m \cup \{ \rho_{nk} \}_{n \in \mathbb Z, \, k = \overline{1, m}}$. Since $\de > 0$ can be arbitrarily small, it follows that 
$$
\rho_{nk} = \rho_{nk}^0 + \varkappa_{nk}, \quad \varkappa_{nk} = o(1), \quad  n \to \pm \iy, \quad k = \overline{1, m}.
$$

{\bf Step 2.} Let us prove that $\{ \varkappa_{nk} \} \in l_2$. Fix a $k \in \{ 1,  \dots, m \}$. Using~\eqref{defW0}, \eqref{V2}, and the Taylor formula, we get
\begin{equation} \label{sm4}
W(\rho_{nk}) = (-1)^n W^0(r_k) + (-1)^n \varkappa_{nk}\dot W^0(r_k) + O(\varkappa_{nk}^2) + K(\rho_{nk}), 
\end{equation}
where $\dot W^0(\rho) = \tfrac{d}{d\rho} W^0(\rho)$. 
Using~\eqref{defK}, we obtain
\begin{align*}
K(\rho_{nk}) & = \int_{-\pi}^{\pi} \mathscr P(t) \exp(i r_k t) \exp(i n t) \, dt \\ & + i \varkappa_{nk} \int_{-\pi}^{\pi} t \mathscr P(t) \exp(i r_k t) \exp(i n t) \, dt + O(\varkappa_{nk}^2) + O\left(n^{-1}\right).
\end{align*}
Thus,
\begin{equation} \label{sm5}
K(\rho_{nk}) = O(\de_{nk}) + O(\varkappa_{nk}^2),
\end{equation}
where $\{ \de_{nk} \} \in l_2$ is some sequence of positive numbers.
Since $\det (W(\rho_{nk})) = 0$, there exists a normalized vector $y_{nk} \in \mathbb C^m$ such that 
\begin{equation} \label{sm6}
    W(\rho_{nk}) y_{nk} = 0.
\end{equation} 
Lemma~\ref{lem:ranks} implies that $(W^0(\rho))^{-1}$ has a simple pole at $\rho = r_k$. Consequently, there exist the matrices $R_{-1}$ and $R_0$ such that
$$
(W^0(\rho))^{-1} = (W^0(r_k) + (\rho - r_k) \dot W^0(r_k) + \dots)^{-1} = \frac{R_{-1}}{\rho - r_k} + R_0 + \dots.
$$
In particular, 
\begin{equation} \label{sm7}
R_{-1} W^0(r_k) = 0, \quad R_{-1} \dot W^0(r_k) + R_0 W^0(r_k) = I.
\end{equation}
Combining~\eqref{sm4}, \eqref{sm5}, and~\eqref{sm6}, we obtain
$$
(\varkappa_{nk}^{-1} R_{-1} + R_0) (W^0(r_k) + \varkappa_{nk} \dot W^0(r_k) + O(\varkappa_{nk}^2) + O(\de_{nk})) y_{nk} = 0.
$$
Using~\eqref{sm7}, we derive
$$
\varkappa_{nk} (y_{nk} + o(1)) = O(\de_{nk}), \quad n \to \pm \iy.
$$
Since $\| y_{nk} \| = 1$, we get $\varkappa_{nk} = O(\de_{nk})$, so $\{ \varkappa_{nk} \} \in l_2$. The above arguments are only valid for $\varkappa_{nk} \ne 0$. The case $\varkappa_{nk} = 0$ is trivial.

\smallskip

{\bf Step 3.}
Using~\eqref{V20}, \eqref{V2}, and the results of steps 1-2, we conclude that the functions $\Delta(\la) = \det(V(\vv))$ and $\Delta^0(\la) = \det(V^0(\vv))$ have the same number of zeros in any sufficiently large disk $\{ |\la| = R \}$ such that $\sqrt R \ne n + r_k$, $n \in \mathbb N$, $k = \overline{1, m}$, and for the zeros $\{ \la_{nk} \}_{(n, k) \in J}$ of $\Delta(\la)$ asymptotics~\eqref{asymptla} holds. Taking Lemma~\ref{lem:mult} into account, we arrive at the claim of the theorem.
\end{proof}

Let us obtain asymptotics of the weight matrices $\{ \al_{nk} \}$. Using~\eqref{relM} and Theorem~\ref{thm:trans}, we get
\begin{gather} \label{ME}
M(\la) = (T_1 + \rho T_1^{\perp}) E(\rho) (\rho^{-1} T_1 + T_1^{\perp}), \\ \nonumber
E(\rho) := (W(\rho))^{-1} U(\rho), \quad U(\rho) = U^0(\rho) + K(\rho),
\end{gather}
where $K(\rho)$ has the form~\eqref{defK}.

Further, we need additional notations. Let $\la_{n_1 k_1} = \la_{n_2 k_2} = \dots = \la_{n_r k_r}$ be a group of multiple eigenvalues maximal by inclusion, $(n_1, k_1) < (n_2, k_2) < \dots < (n_r, k_r)$. Clearly, $\al_{n_1 k_1} = \al_{n_2 k_2} = \dots = \al_{n_r k_r}$. Define $\al'_{n_1 k_1} := \al_{n_1 k_1}$, $\al_{n_j k_j} := 0$, $j = \overline{2, r}$. We obtain the sequences of matrices $\{ \al'_{nk} \}_{(n, k) \in J}$. Below the notation $\{ K_{nk} \}$ is used for various matrix sequences such that $\{ \| K_{nk}\| \} \in l_2$.

\begin{thm} \label{thm:asymptal}
The weight matrices are Hermitian non-negative definite: $\al_{nk} = \al_{nk}^{\dagger} \ge 0$, $(n, k) \in J$. For each $(n, k) \in J$, $\rank(\al_{nk})$ equals the multiplicity of the eigenvalue $\la_{nk}$. Furthermore, the asymptotic formula holds:
\begin{equation} \label{asymptal}
\al_n^{(k)} := \sum_{\substack{s = \overline{1, m} \\ r_s = r_k}} \al'_{ns} = \frac{2}{\pi} (T_1 + n T_1^{\perp}) (A_k + K_{nk}) (T_1 + n T_1^{\perp}), \quad n \ge 1, \quad k = \overline{1, m},
\end{equation}
where $A_k$ is defined in \eqref{defAk}. 
\end{thm}

\begin{proof}
It remains to prove \eqref{asymptal}, since all the other properties have been proved in Section~2. In particular, the non-negative definiteness of $\al_{nk}$ follows from~\eqref{sym1}.

Fix a $k \in \{ 1, \dots, m \}$ and choose a sufficiently small $\de > 0$. For sufficiently large $n$, the Residue Theorem and~\eqref{ME} imply
\begin{equation} \label{sm8}
\al_n^{(k)} = \frac{1}{2\pi i} \oint\limits_{|\sqrt{\la} - \rho_{nk}^0| = \de} M(\la) \, d\la = \frac{1}{2 \pi i} \oint\limits_{|\rho - \rho_{nk}^0| = \de} 2 \rho (T_1 + \rho T_1^{\perp}) E(\rho) (\rho^{-1} T_1 + T_1^{\perp})\, d\rho.     
\end{equation}
Using~\eqref{defAk}, we obtain
\begin{equation} \label{sm9}
\frac{1}{2 \pi i} \oint\limits_{|\rho - \rho_{nk}^0| = \de} 2 \rho (T_1 + \rho T_1^{\perp}) E^0(\rho) (\rho^{-1} T_1 + T_1^{\perp})\, d\rho = \frac{2}{\pi} (T_1 + n T_1^{\perp}) (A_k + O(n^{-1})) (T_1 + n T_1^{\perp}).
\end{equation}
Let us estimate the difference
$$
E(\rho) - E^0(\rho) = ((W(\rho))^{-1} - (W^0(\rho))^{-1}) U(\rho) + (W^0(\rho))^{-1} (U^0(\rho) - U(\rho)).
$$
Note that
$$
(W(\rho))^{-1} - (W^0(\rho))^{-1} = (W^0(\rho))^{-1} (I + (W^0(\rho))^{-1} K(\rho))^{-1} - I).
$$
Taking~\eqref{W0below} into account, we get the estimate
$$
\| E(\rho) - E^0(\rho) \| \le C \| K(\rho) \| \exp(-|\tau|\pi), \quad \rho \in G_{\de}.
$$
Represent $\rho \colon |\rho - \rho_{nk}^0| = \de$ in the form $\rho = n + r_k + z$, $|z| = \de$. In view of~\eqref{defK}, the sum 
$$
\sum_{n = 1}^{\iy} \| K(n + r_k + z) \|^2
$$
is bounded uniformly on $|z| = \de$. Consequently, we obtain
\begin{equation} \label{sm10}
\frac{1}{2 \pi i} \oint\limits_{|\rho - \rho_{nk}^0| = \de} 2 \rho (T_1 + \rho T_1^{\perp}) (E(\rho) - E^0(\rho)) (\rho^{-1} T_1 + T_1^{\perp})\, d\rho = (T_1 + n T_1^{\perp}) K_{nk} (T_1 + n T_1^{\perp}), 
\end{equation}
where $\{ \| K_{nk} \| \} \in l_2$. Combining \eqref{sm8}, \eqref{sm9}, and~\eqref{sm10}, we arrive at~\eqref{asymptal}.
\end{proof}

In addition, we obtain estimates for $\vv(x, \la)$ and $\Phi(x, \la)$.

\begin{lem} \label{lem:estsol}
The following relations hold
\begin{gather} \label{estphi}
    \left.
   \begin{array}{r}
    \vv^{[\nu]}(x, \la) = O(\rho^{\nu - 1}\exp(|\tau|x) (\rho T_1 + T_1) \\
    \vv^{[\nu]}(x, \la) - \vv^{0[\nu]}(x, \la) = o(\rho^{\nu - 1}\exp(|\tau|x) (\rho T_1 + T_1) 
    \end{array}
    \right\}
    \\ \label{estPhi}
    \left.
    \begin{array}{r}
    \Phi^{[\nu]}(x, \la) = O(\rho^{\nu - 1} \exp(-|\tau|x) (T_1 + \rho T_1^{\perp}) \\
    \Phi^{[\nu]}(x, \la) - \Phi^{0[\nu]}(x, \la) = o(\rho^{\nu - 1} \exp(-|\tau|x) (T_1 + \rho T_1^{\perp})
    \end{array}
    \right\}
    \quad \rho \in G_{\de},
\end{gather}
for some $\de > 0$, each fixed $x \in [0, \pi]$, $\nu = 0, 1$ as $|\rho| \to \iy$. Here $y^{[0]} := y$.
\end{lem}

\begin{proof}
Fix an $x \in [0, \pi]$. Using~\eqref{phiCS} and Theorem~\ref{thm:trans}, we get
\begin{align*}
    \vv(x, \la) & = (\cos \rho x \, T_1 + \sin \rho x \, T_1^{\perp} + o(\exp(|\tau|x))(T_1 + \rho^{-1} T_1^{\perp}), \\
    \vv^{[1]}(x, \la) & = (-\sin \rho x \, T_1 + \cos \rho x \, T_1^{\perp} + o(\exp(|\tau|x)) (\rho T_1 + T_1^{\perp}),
\end{align*}
as $|\rho| \to \iy$. These asymptotics yield \eqref{estphi}. Similar asymptotics hold for the solution $\Psi(x, \la)$ appearing in~\eqref{PhiPsi}:
\begin{align*}
    \Psi(x, \la) & = (\cos \rho (\pi - x) \, T_2 - \sin \rho (\pi - x) \, T_2^{\perp} + o(\exp(|\tau|(\pi - x)))(T_2 + \rho^{-1} T_2^{\perp}), \\
    \Psi^{[1]}(x, \la) & = (\sin \rho (\pi - x) \, T_2 + \cos \rho (\pi - x) \, T_2^{\perp} + o(\exp(|\tau|(\pi - x))) (\rho T_2 + T_2^{\perp}),
\end{align*}
as $|\rho| \to \iy$. These asymptotics imply
\begin{equation} \label{estPsi}
\left.
\begin{array}{r}
\Psi^{[\nu]}(x, \la) = O(\rho^{\nu - 1}\exp(|\tau|(\pi - x))) (\rho T_2 + T_2^{\perp}) \\
\Psi^{[\nu]}(x, \la) - \Psi^{0[\nu]}(x, \la) = o(\rho^{\nu - 1}\exp(|\tau|(\pi - x))) (\rho T_2 + T_2^{\perp})
\end{array} \right\}
\end{equation}
as $|\rho| \to \iy$, $\nu = 0, 1$.

Analogously to~\eqref{V2}, one can obtain
$$
V_1(\Psi) = (\rho T_1 + T_1^{\perp}) ((W^0(\rho))^{\dagger} + K(\rho)) (T_2 + \rho^{-1} T_2^{\perp}).
$$
Using~\eqref{W0below}, we get
\begin{equation} \label{VPsi}
\left.
\begin{array}{r}
(V_1(\Psi))^{-1} = (T_2 + \rho T_2^{\perp}) O(\exp(-|\tau|\pi)) (\rho^{-1} T_1 + T_1^{\perp}) \\
(V_1(\Psi))^{-1} - (V_1^0(\Psi^0))^{-1} = (T_2 + \rho T_2^{\perp}) o(\exp(-|\tau|\pi)) (\rho^{-1} T_1 + T_1^{\perp})
\end{array}\right\}
\end{equation}
as $|\rho| \to \iy$, $\rho \in G_{\de}$.
Relation \eqref{PhiPsi} implies
\begin{equation} \label{difPhi}
\Phi(x, \la) - \Phi^0(x, \la) = (\Psi(x, \la) - \Psi^0(x, \la)) (V_1(\Psi))^{-1} + \Psi^0(x, \la) ((V_1(\Psi))^{-1} - (V_1^0(\Psi))^{-1}).
\end{equation}
Using~\eqref{PhiPsi} and \eqref{estPsi}-\eqref{difPhi}, we arrive at~\eqref{estPhi}.
\end{proof}

\section{Completeness and Riesz-basis property}

In this section, we study the completeness and the Riesz-basis property of a special sequence of vector functions constructed by the spectral data $\{ \la_{nk}, \al_{nk} \}_{(n, k) \in J}$. Such sequences play an important role in inverse problem theory for matrix Sturm-Liouville operators (see \cite{CK09, MT09, Bond19}).

Consider a group of multiple eigenvalues $\la_{n_1 k_1} = \la_{n_2 k_2} = \dots = \la_{n_r k_r}$ maximal by inclusion. Lemma~\ref{lem:ranks} implies $\rank(\al_{n_1 k_1}) = r$. Define the matrices
$$
T_{nk} := \left\{ \begin{array}{ll}
                T_1 + \rho_{nk} T_1^{\perp}, \quad & \rho_{nk} \ne 0, \\
                I, \quad & \rho_{nk} = 0,
          \end{array} \right.
\qquad          
B_{nk} := \frac{\pi}{2} T_{nk}^{-1} \al_{nk} T_{nk}^{-1}.
$$
Clearly, $\Ran B_{n_1 k_1}$ is an $r$-dimensional subspace in $\mathbb C^m$. Choose an orthonormal basis $\{ \mathcal E_{n_j k_j} \}_{j = 1}^r$ in this subspace. This choice is non-unique. The assertions below are valid for any choice of the basis. Thus, we have defined the vector sequence $\{ \mathcal E_{nk} \}_{(n, k) \in J}$. Consider the sequence of vector functions
\begin{equation} \label{defY}
\mathcal Y := \{ Y_{nk} \}_{(n, k) \in J}, \quad
Y_{nk}(x) := \left\{ \begin{array}{ll}
                (\cos (\rho_{nk}x) T_1 + \sin (\rho_{nk}x) T_1^{\perp}) \mathcal E_{nk}, \quad & \rho_{nk} \ne 0, \\
                (T_1 + x T_1^{\perp}) \mathcal E_{nk}, \quad & \rho_{nk} = 0.
            \end{array}\right.
\end{equation}

\begin{thm} \label{thm:complete}
The sequence $\mathcal Y$ is complete in $L_2((0, \pi); \mathbb C^m)$.
\end{thm}

\begin{proof}
Let a vector function $h \in L_2((0, \pi); \mathbb C^m)$ be such that
$(h, Y_{nk}) = 0$ for all $(n, k) \in J$. This implies
$$
\int_0^{\pi} (h(x))^{\dagger} \left( \cos (\rho_{nk} x) T_1 + \frac{\sin (\rho_{nk} x)}{\rho_{nk}} T_1^{\perp}\right) \al_{nk} \, dx = 0, \quad (n, k) \in J.
$$
Consequently, the row-vector function
$$
\ga(\la) := \int_0^{\pi} (h(x))^{\dagger} \left( \cos \rho x \, T_1 + \frac{\sin \rho x}{\rho} T_1^{\perp} \right) \, dx
$$
is entire in $\la$ and has the following properties:

\smallskip

(i) $\ga(\la_{nk}) \al_{nk} = 0$, $(n, k) \in J$.

\smallskip

(ii) $\ga(\rho^2) (T_1 + \rho T_1^{\perp}) = o(\exp(|\tau|\pi))$, $|\rho| \to \iy$.

\smallskip

If the multiplicity of the eigenvalue $\la_{nk}$ is $m_{nk}$, then 
$$
\rank (\al_{nk}) = m_{nk}, \quad \rank (V_2(\vv(x, \la_{nk})) = m - m_{nk}.
$$
Therefore, relation~\eqref{Val} and property (i) imply 
$$
\ga(\la_{nk}) = D_{nk} V_2(\vv(x, \la_{nk})), \quad D_{nk}^{\dagger} \in \mathbb C^m, \quad (n, k) \in J.
$$
Consequently, the vector function $F(\la) := \ga(\la) (V_2(\vv(x, \la)))^{-1}$ is entire in $\la$. It follows from \eqref{W0below} and
\eqref{V2} that
$$
(V_2(\vv))^{-1} = (T_1 + \rho T_1^{\perp}) O(\exp(-|\tau|\pi)) (\rho^{-1} T_2 + T_2^{\perp}), \quad \rho \in G_{\de}, \quad |\rho| \to \iy.
$$
Using this estimate and property (ii), we obtain $F(\rho^2) = o(1)$, $\rho \in G_{\de}$, $|\rho| \to \iy$. Liouville's Theorem yields $F(\la) \equiv 0$, so $\ga(\la) \equiv 0$. Hence, $h = 0$ in $L_2((0, \pi); \mathbb C^m)$, so the sequence $\mathcal Y$ is complete.
\end{proof}

In particular, Theorem~\ref{thm:complete} yields that the following sequence $\mathcal Y^0$ related to the problem $L^0$ is complete in $L_2((0, \pi); \mathbb C^m)$:
$$
\mathcal Y^0 := \{ Y_{nk}^0 \}_{(n, k) \in J}, \quad 
Y_{nk}^0 := (\cos (n + r_k) x \, T_1 + \sin (n + r_k) x \, T_1^{\perp}) \mathcal E_k^0.
$$
Here $\{ \mathcal E_s^0 \}_{s \in J_k}$ is a fixed orthonormal basis in $\Ran A_k$ for $k \in \mathcal J$, $J_k := \{ s = \overline{1, m} \colon r_s = r_k \}$, $\mathcal J := \{ 1 \} \cup \{ s = \overline{2, m} \colon r_s \ne r_{s-1} \}$. We additionally require $T_1 \mathcal E_k^0 = 0$ for $k = \overline{1, p^{\perp}}$ and $T_1^{\perp} \mathcal E_k^0 = 0$ for $k = \overline{p^{\perp} + 1, p^{\perp} + p}$. The latter requirements can always be achieved because of~\eqref{A1}.
Thus, $\{ \mathcal E_k^0 \}_{k = 1}^m$ is an orthonormal basis in $\mathbb C^m$.

Our next goal is to show that $\mathcal Y$ is a Riesz basis. We will prove this fact for a sequence of a more general form, not necessarily related to the problem $L$. 
Let $T_1, T_2 \in \mathbb C^{m \times m}$ be arbitrary orthogonal projection matrices.
Suppose that $J$ and $\{ r_k \}_{k = 1}^m$ are defined as in Theorem~\ref{thm:asymptla}.
Let $\{ \rho_{nk} \}_{(n, k) \in J}$ be arbitrary complex numbers satisfying asymptotics~\eqref{asymptla}.
Let $\{ B_{nk} \}_{(n, k) \in J}$ be arbitrary matrices from $\mathbb C^{m \times m}$ such that $B_{nk} = B_{nk}^{\dagger} \ge 0$, $(n, k) \in J$. For any group of multiple values $\rho_{n_1 k_1} = \rho_{n_2 k_2} = \dots = \rho_{n_r k_r}$ maximal by inclusion, $(n_1, k_1) < (n_2, k_2) < \dots < (n_r, k_r)$, we assume that $B_{n_1 k_1} = B_{n_2 k_2} = \dots = B_{n_r k_r}$ and $\rank (B_{n_1 k_1}) = r$. Denote $B'_{n_1 k_1} = B_{n_1 k_1}$, $B'_{n_j k_j} = 0$, $j = \overline{2, r}$. Suppose that the following asymptotic relation holds
\begin{equation} \label{asymptB}
B_n^{(k)} := \sum_{\substack{s = \overline{1, m} \\ r_s = r_k}} B'_{ns} = A_k + K_{nk}, \quad \{ \| K_{nk} \| \} \in l_2, \quad (n, k) \in J,
\end{equation}
where $\{ A_k \}_{k = 1}^m$ are the orthogonal projection matrices defined by~\eqref{defAk}.
By using these data $\{ \rho_{nk}, B_{nk} \}_{(n, k) \in J}$, choose the basis $\{ \mathcal E_{nk} \}$ and construct the sequence $\mathcal Y$ by~\eqref{defY}. Then the following theorem holds.

\begin{thm} \label{thm:Riesz}
If the sequence $\mathcal Y$ is complete in $L_2((0, \pi); \mathbb C^m)$, then it is a Riesz basis in $L_2((0, \pi); \mathbb C^m)$.
\end{thm}

For the proof of Theorem~\ref{thm:Riesz}, we need the following propositions. (Proposition~\ref{prop:sing} is \cite[Theorem~2.5.3]{GV96} and Proposition~\ref{prop:ineq} follows from~\cite[Theorem~3.6.6]{Christ03}).

\begin{prop} \label{prop:sing}
Let $A \in \mathbb C^{m \times n}$ and $k < \rank(A)$. Denote by $s_1 \ge s_2 \ge \dots \ge s_{\min(n, m)}$
the singular values of $A$. Then
$$
    \min_{\rank(B) = k} \| A - B \| = s_{k + 1}.
$$
\end{prop}

\begin{prop} \label{prop:ineq}
Let $\{ f_n \}_{n = 1}^{\iy}$ be a sequence in a Hilbert space $H$. The sequence $\{ f_n \}_{n = 1}^{\iy}$ is a Riesz basis in $H$ if and only if it is complete in $H$ and there exist constants $M_1, M_2 > 0$ such that for every finite scalar sequence $\{ b_n \}$ one has
\begin{equation} \label{ineq}
M_1 \sum |b_n|^2 \le \left\| \sum b_n f_n \right\|^2 \le M_2 \sum |b_n|^2.
\end{equation}
\end{prop}

\begin{proof}[Proof of Theorem~\ref{thm:Riesz}]
Instead of $\mathcal Y$, consider the sequence 
$$
\mathcal Y_N := \{ Y_{nk}^0 \}_{(n, k) \in J, \, n \le N} \cup \{ \tilde Y_{nk} \}_{n > N, \, k = \overline{1, m}}, \quad N \in \mathbb N,
$$
where
$$
\tilde Y_{nk}(x) := (\cos (n + r_k) x \, T_1 + \sin (n + r_k) x \, T_1^{\perp}) \tilde {\mathcal E}_{nk}, \quad \tilde {\mathcal E}_{nk} := A_k \mathcal E_{nk}.
$$
We will show that $\mathcal Y_N$ is quadratically close to $\mathcal Y$ and, for sufficiently large $N$, the sequence $\mathcal Y_N$ is a Riesz basis in $L_2((0, \pi); \mathbb C^m)$. This will imply that $\mathcal Y$ is also a Riesz basis.

\smallskip

\textbf{Step 1.} 
Let us prove that $\{ \| \mathcal E_{nk} - \tilde {\mathcal E}_{nk}\| \} \in l_2$. This will imply that $\{ \| Y_{nk} - \tilde Y_{nk} \| \} \in l_2$, that is, $\mathcal Y_N$ is quadratically close to $\mathcal Y$.
Fix $k \in \mathcal J$.
Obviously, $\mathcal E_{nk} - \tilde {\mathcal E}_{nk} = A_k^{\perp} \mathcal E_{nk}$, where $A_k^{\perp} := I - A_k$.  Let $\mathcal E_n^{(k)} \in \mathbb C^{m \times |J_k|}$ be the matrix consisting of the columns $\{ \mathcal E_{ns} \}_{s \in J_k}$. By the definitions of $\{ \mathcal E_{ns} \}$ and $B_n^{(k)}$, for each sufficiently large $n$, there exists the matrix $w_n^{(k)} \in \mathbb C^{|J_k| \times m}$ such that $B_n^{(k)} = \mathcal E_n^{(k)} w_n^{(k)}$. Asymptotics~\eqref{asymptB} implies $B_n^{(k)} = O(1)$ as $n \to \iy$. Since $B_{ns} \ge 0$, we also have $B_{ns} = O(1)$, $s \in J_k$, $n \to \iy$. The columns of $\mathcal E_n^{(k)}$ are normalized vectors, so $w_n^{(k)} = O(1)$ as $n \to \iy$. The asymptotic relation~\eqref{asymptB} implies $\mathcal E_n^{(k)} w_n^{(k)} = A_k + K_{nk}$. Hence,
$$
A_k^{\perp} \mathcal E_n^{(k)} w_n^{(k)} (w_n^{(k)})^{\dagger} = K_{nk}.
$$
Consider the minimal singular value $s_{min}(w_n^{(k)})$ of the matrix $w_n^{(k)}$. If $s_{\min}(w_n^{(k)}) \ge \de > 0$ for all sufficiently large $n$, then $\| (w_n^{(k)} (w_n^{(k)})^{\dagger})^{-1} \| \le \frac{1}{\de^2}$. Therefore, $A_k^{\perp} \mathcal E_n^{(k)} = K_{nk}$, so $\{ \| A_k^{\perp} \mathcal E_{nk} \| \} \in l_2$.

\smallskip

\textbf{Step 2.} Let us prove that $s_{min}(w_n^{(k)}) \ge \de > 0$ for all sufficiently large $n$ and a fixed $k \in \mathcal J$. Suppose that, on the contrary, there exists a subsequence $\{ n_j \}$ such that $s_{min}(w_{n_j}^{(k)}) \to 0$ as $j \to \iy$. By virtue of Proposition~\ref{prop:sing}, there exist matrices $\tilde w_{n_j}^{(k)} \in \mathbb C^{|J_k| \times m}$ such that $\rank(\tilde w_{n_j}^{(k)}) < |J_k|$ and $\| w_{n_j}^{(k)} - \tilde w_{n_j}^{(k)} \| \to 0$ as $j \to \iy$. Denote $\tilde B_{n_j}^{(k)} := \mathcal E_{n_j}^{(k)} \tilde w_{n_j}^{(k)}$. Obviously, $\rank (B_{n_j}^{(k)}) < |J_k|$. Note that
$$
\| B_{n_j}^{(k)} - \tilde B_{n_j}^{(k)} \| \le \| \mathcal E_n^{(k)} \| \| w_{n_j}^{(k)} - \tilde w_{n_j}^{(k)} \| \to 0, \quad j \to \iy.
$$
This together with~\eqref{asymptB} imply $\| \tilde B_{n_j}^{(k)} - A_k \| \to 0$ as $j \to \iy$. Proposition~\ref{prop:sing} yields $s_{|J_k|}(A_k) = 0$, but $\rank (A_k) = J_k$. This contradiction concludes the proof.

\smallskip

\textbf{Step 3.} Let us prove that the sequence $\mathcal Y_N$ is complete in $L_2((0, \pi); \mathbb C^m)$ for each sufficiently large $N$. For $k \in \mathcal J$, let $\tilde {\mathcal E}_n^{(k)} \in \mathbb C^{m \times |J_k|}$ be the matrix consisting of the columns $\{ \tilde {\mathcal E}_{ns} \}_{s \in J_k}$. Similarly to step~2, one can show that $s_{min}(\tilde {\mathcal E}_n^{(k)}) \ge \de > 0$ for all sufficiently large $n$, so $\rank(\tilde {\mathcal E}_n^{(k)}) = |J_k|$. Since $\tilde {\mathcal E}_{ns} \in \Ran(A_k)$, $s \in J_k$, it follows that the vectors $\{ \mathcal E_s^0 \}_{s \in J_k}$ are linear combinations of $\{ \tilde {\mathcal E}_{ns} \}_{s \in J_k}$. Consequently, vector functions $\{ Y_{ns}^0 \}_{s \in J_k}$ are linear combinations of $\{ \tilde Y_{ns} \}_{s \in J_k}$ for each sufficiently large fixed $n$ and each $k \in \mathcal J$. By Theorem~\ref{thm:complete}, the sequence $\mathcal Y^0 = \{ Y_{nk}^0 \}_{(n, k) \in J}$ is complete in $L_2((0, \pi); \mathbb C^m)$. Therefore, the sequence $\mathcal Y_N$ is also complete for sufficiently large $N$.

\smallskip

\textbf{Step 4.} Let us prove that the sequence $\mathcal Y_N$ is a Riesz basis in $L_2((0, \pi); \mathbb C^m)$ for sufficiently large $N$, relying on the completeness of this sequence (step~3) and on Proposition~\ref{prop:ineq}. It remains to prove inequality~\eqref{ineq}, which takes the form
\begin{equation} \label{ineq1}
 M_1 \sum_{n, k} |b_{nk}|^2 \le \left\| \sum_{n, k} b_{nk} Y^N_{nk} \right\|^2 \le M_2 \sum_{n, k} |b_{nk}|^2, \quad 
 Y^N_{nk} := \begin{cases}
                Y_{nk}^0, \quad n \le N, \\
                \tilde Y_{nk}, \quad n > N.
            \end{cases}    
\end{equation}
The right-hand side of this inequality is obvious, since $\| Y^N_{nk} \| \le \pi$ for all $(n, k) \in J$. It follows from~\eqref{propAk} that $(Y^N_{nk}, Y^N_{ls}) = 0$ if $n \ne l$ or $r_k \ne r_s$. Hence,
$$
\left\| \sum_{n, k} b_{nk} Y^N_{nk} \right\|^2 = \sum_n \sum_{k \in \mathcal J} \left\| \sum_{s \in J_k} b_{ns} Y^N_{ns} \right\|^2
$$
Consequently, in order to prove the left-hand side of~\eqref{ineq1}, it is sufficient to show that
$$
\left \| \sum_{s \in J_k} a_s \tilde {\mathcal E}_{ns}\right\|^2 \ge \de^2 \sum_{s \in J_k} |a_s|^2, \quad 
\left \| \sum_{s \in J_k} a_s \mathcal E_s^0 \right\|^2 \ge \de^2 \sum_{s \in J_k} |a_s|^2, 
$$
for any $\{ a_s \}_{s \in J_k}$, all $n > N$, $k \in \mathcal J$, and some $\de > 0$. The inequality for $\{ \tilde {\mathcal E}_{ns} \}$ follows from the estimate $s_{min}(\tilde {\mathcal E}_n^{(k)}) \ge \de$, which is valid for sufficiently large $n$. The inequality for $\{ \mathcal E_s^0 \}$ is obvious, since these vectors form an orthonormal basis. Thus, we have proved inequality \eqref{ineq1}, which yields the claim.
\end{proof}

\begin{remark}
Similarly to Theorem~\ref{thm:complete},
it can be proved that the sequence $\mathcal F := \{ \vv(x, \la_{nk}) T_{nk} \mathcal E_{nk} \}_{(n, k) \in J}$ of the vector eigenfunctions of $L$ is complete in $L_2((0, \pi); \mathbb C^m)$. Since $\mathcal F$ is quadratically close to $\mathcal Y$, it follows that $\mathcal F$ is a Riesz basis in $L_2((0, \pi); \mathbb C^m)$.
\end{remark}

\section{Inverse problem}

In this section, we consider the problem $L = L(\sigma, T_1, T_2, H_2)$ of the form~\eqref{eqv}-\eqref{bc2} with $H_1 = 0$ and prove the uniqueness theorem for the following inverse problem.

\begin{ip}
Given the spectral data $\{ \la_{nk}, \al_{nk} \}_{(n, k) \in J}$, find $\sigma$, $T_1$, $T_2$, $H_2$.
\end{ip}

Along with $L$, consider another boundary value problem $\tilde L = L(\tilde \sigma, \tilde T_1, \tilde T_2, \tilde H_2)$ of the same form but with different coefficients. We agree that if a symbol $\ga$ denotes an object related to $L$, then the symbol $\tilde \ga$ with tilde denotes the similar object related to $\tilde L$. Note that the quasi-derivatives for these two problems are supposed to be different: $Y^{[1]} = Y' - \sigma Y$ for $L$ and $Y^{[1]} = Y' - \tilde \sigma Y$ for $\tilde L$. The goal of this section is to prove the following uniqueness theorem.

\begin{thm} \label{thm:uniq}
If $\la_{nk} = \tilde \la_{nk}$, $\al_{nk} = \tilde \al_{nk}$, $(n, k) \in J$, $J = \tilde J$, then 
\begin{equation} \label{transL}
\sigma(x) = \tilde \sigma(x) + H_1^{\diamond} \:\: \text{a.e. on} \:\: (0, \pi), \quad T_1 = \tilde T_1, \quad T_2 = \tilde T_2, \quad H_2 = \tilde H_2 - T_2 H_1^{\diamond} T_2,
\end{equation}
where 
\begin{equation} \label{H1perp}
    H_1^{\diamond} = (H_1^{\diamond})^{\dagger} = T_1^{\perp} H_1^{\diamond} T_1^{\perp}.
\end{equation}
Thus, the spectral data $\{ \la_{nk}, \al_{nk} \}_{(n, k) \in J}$ uniquely specify the problem $L$ up to a transform~\eqref{transL} given by an arbitrary matrix $H_1^{\diamond}$ satisfying~\eqref{H1perp}.
\end{thm}

Theorem~\ref{thm:uniq} is a natural generalization of the known uniqueness results for $m = 1$ in the cases of Dirichlet-Dirichlet, Dirichlet-Robin, Robin-Dirichlet, and Robin-Robin boundary conditions (see \cite{HM-sd, Gul19}).

The inverse is also true: any two problems $L$ and $\tilde L$ satisfying~\eqref{transL} with an arbitrary matrix $H_1^{\diamond}$ of the form~\eqref{H1perp} have equal spectral data. Indeed, in this case, we obtain $M(\la) = \tilde M(\la) + H_1^{\diamond}$, so the poles and the residues of $M(\la)$ and $\tilde M(\la)$ coincide. This together with Theorem~\ref{thm:uniq} immediately imply the following result.

\begin{thm}
If $M(\la) \equiv \tilde M(\la)$, then $\sigma(x) = \tilde \sigma(x)$ a.e. on $(0, \pi)$, $T_1 = \tilde T_1$, $T_2 = \tilde T_2$, $H_2 = \tilde H_2$. Thus, the Weyl matrix $M(\la)$ uniquely specifies the problem $L(\sigma, T_1, T_2, H_2)$.
\end{thm}

For the regular potential $Q \in L_1((0, \pi); \mathbb C^{m \times m})$, Theorem~\ref{thm:uniq} implies the uniqueness without ambiguity. Indeed, consider the boundary value problem $X = X(Q, T_1, T_2, G_1, G_2)$ for equation~\eqref{eqvQ} with $Q \in L_1((0, \pi); \mathbb C^{m \times m})$ and with the boundary conditions
\begin{align*}
   & V_1(Y) = T_1 (Y'(0) - G_1 Y(0)) - T_1^{\perp} Y(0) = 0, \\
   & V_2(Y) = T_2 (Y'(\pi) - G_2 Y(\pi)) - T_2^{\perp} Y(\pi) = 0,
\end{align*}
where $G_j \in \mathbb C^{m \times m}$, $G_j = G_j^{\dagger} = T_j G_j T_j$, $j = 1, 2$.
These conditions are equivalent to~\eqref{bc1} and~\eqref{bc2} when $G_1 = T_1 \sigma(0) T_1$ (if $H_1 = 0$), $G_2 = H_2 + T_2 \sigma(\pi) T_2$. Instead of~\eqref{bc3}, we consider the condition
$$
  V_1^{\perp}(Y) = T_1^{\perp} Y'(0) + T_1 Y(0) = 0.
$$
This condition is equivalent to \eqref{bc3} if $T_1^{\perp} \sigma(0) T_1^{\perp} = 0$, so it fixes the constant $H_1^{\diamond}$ in Theorem~\ref{thm:uniq}. The spectral data $\{ \la_{nk}, \al_{nk} \}_{(n, k) \in J}$ of $X$ are defined similarly to the spectral data of $L$, by using the new $V_1$, $V_2$, and $V_1^{\perp}$. Along with $X$, consider the problem $\tilde X = X(\tilde Q, \tilde T_1, \tilde T_2, \tilde G_1, \tilde G_2)$. For the spectral data of $X$ and $\tilde X$, the following uniqueness theorem directly follows from Theorem~\ref{thm:uniq}. 

\begin{thm} \label{thm:reg}
If $\la_{nk} = \tilde \la_{nk}$, $\al_{nk} = \tilde \al_{nk}$, $(n, k) \in J$, $J = \tilde J$, then $Q(x) = \tilde Q(x)$ a.e. on $(0, \pi)$, $T_1 = \tilde T_1$, $T_2 = \tilde T_2$, $G_1 = \tilde G_1$, $G_2 = \tilde G_2$.
\end{thm}

Theorem~\ref{thm:reg} improves the results of \cite{Xu19} and generalizes the uniqueness results from \cite{Yur06-uniq}.

Before the proof of Theorem~\ref{thm:uniq}, let us discuss the construction the Weyl matrix $M(\la)$ by using the spectral data. Fix an arbitrary real $\om > 0$ and put $\be(\la) = \frac{\la}{\la^2 + \om^2}$. Using~\eqref{asymptla} and~\eqref{asymptal}, we get
$$
\left( \frac{1}{\la - \la_{nk}} + \be(\la_{nk})\right) \al_{nk} = O(n^{-2}), \quad n \to \iy.
$$
Consequently, the series
$$
\mathfrak M(\la) := \sum_{(n, k) \in J} \left( \frac{1}{\la - \la_{nk}} + \be(\la_{nk})\right) \al'_{nk}
$$
converges absolutely and uniformly by $\la$ on compact sets. 
The following lemma is proved similarly to \cite[Theorem~1]{BSY13}.

\begin{lem} \label{lem:Weyl}
$M(\la) = \mathfrak M(\la) + C_*$, where $C_* \in \mathbb C^{m \times m}$ is a constant matrix.
\end{lem}

Using~\eqref{ME0}, \eqref{defW0}, and~\eqref{defU0}, we obtain the asymptotic formula
$$
M^0(-\tau^2) = \tau T_1^{\perp} + o(1), \quad \tau \to +\iy.
$$
Consequently, the constant matrix $C_*^0$ for $M^0(\la)$ can be found as follows:
\begin{equation} \label{C00}
C^0_* = \lim_{\tau \to +\iy} (\tau T_1^{\perp} - \mathfrak M^0(-\tau^2)).
\end{equation}
In the general case, $M(\la)$ is recovered from the spectral data uniquely up to a constant matrix of the form~\eqref{H1perp}.

Proceed to reconstruction of $T_1$ and $T_2$ by the spectral data. 
Suppose that $\{ \la_{nk}, \al_{nk} \}_{(n, k) \in J}$ are given. Relations~\eqref{sumAk} and~\eqref{asymptal} imply
\begin{equation} \label{relT1}
T_1^{\perp} = \lim_{n \to \iy} n^{-2} \left( \sum_{k = 1}^m \al'_{nk} \right), \quad T_1 = I - T_1^{\perp}.
\end{equation}
Thus, $T_1$ is found. Using~\eqref{asymptla} and~\eqref{asymptal}, we get
\begin{equation} \label{rkAk}
r_k = \lim_{n \to \iy} (\sqrt{\la_{nk}} - n), \quad 
A_k = \frac{\pi}{2}\lim_{n \to \iy} (T_1 + n^{-1} T_1^{\perp}) \al_n^{(k)} (T_1 + n^{-1} T_1^{\perp}), \quad k = \overline{1, m}.
\end{equation}
Construct the spectral data $\{ \la_{nk}^0, \al_{nk}^0 \}_{(n, k) \in J}$ of $L^0$ by the formulas
\begin{equation} \label{sd0}
\rho_{nk}^0 = n + r_k, \quad \la_{nk}^0 = (\rho_{nk}^0)^2, \quad \al_{nk}^0 = \begin{cases}\frac{2}{\pi} (T_1 + \rho_{nk}^0 T_1^{\perp}) A_k (T_1 + \rho_{nk}^0 T_1^{\perp}), \quad \rho_{nk}^0 \ne 0, \\
\frac{1}{\pi}T_1 A_k T_1, \quad \rho_{nk}^0 = 0.
\end{cases}
\end{equation}
Using Lemma~\ref{lem:Weyl} and~\eqref{C00}, one can find $M^0(\la)$. According to~\eqref{ME0}, we have
\begin{equation} \label{relE0}
    E^0(\rho) = (T_1 + \rho^{-1} T_1^{\perp}) M^0(\rho^2) (\rho T_1 + T_1^{\perp}).
\end{equation}
On the other hand, \eqref{defW0}, \eqref{defU0}, and~\eqref{ME0} imply 
\begin{equation} \label{eqE0}
E^0(\rho) = (A \tan (\rho \pi) + B)^{-1} (A - \tan(\rho \pi) B), \quad 
A := T_2 T_1 + T_2^{\perp} T_1^{\perp}, \quad B := T_2^{\perp} T_1 - T_2 T_1^{\perp}.
\end{equation}

\begin{lem} \label{lem:AB}
Consider the equation
\begin{equation} \label{eqAB}
    (t A + B)^{-1} (A - tB) = E,
\end{equation}
with respect to unknown matrices $A, B \in \mathbb C^{m \times m}$. It is supposed that $E \in \mathbb C^{m \times m}$ and $t \in \mathbb C$ are known, $t \ne \pm i$, and $\det(t A + B) \ne 0$.
Equation~\eqref{eqAB} has the solution $\tilde A = E + t I$, $\tilde B = I - t E$ unique up to multiplication by a non-singular matrix: $\tilde A = DA$, $\tilde B = DB$, $\det(D) \ne 0$.
\end{lem}

Lemma~\ref{lem:AB} is proved by direct calculations.

Fix $\rho_* \ne \rho_{nk}^0$, $n \in \mathbb Z$, $k = \overline{1, m}$, and apply Lemma~\ref{lem:AB} to \eqref{eqE0} with $t := \tan(\rho_* \pi)$. This yields
\begin{align*}
    & D (T_2 T_1 + T_2^{\perp} T_1^{\perp}) = E^0(\rho_*) + \tan (\rho_* \pi) I, \\
    & D (T_2^{\perp} T_1 - T_2 T_1^{\perp}) = I - \tan (\rho_* \pi) E^0(\rho_*),
\end{align*}
where $D \in \mathbb C^{m \times m}$ is an unknown non-singular matrix. Hence,
\begin{equation} \label{DT2}
D T_2 = (E^0(\rho_*) + \tan (\rho_* \pi) I) T_1 + (\tan(\rho_* \pi) E^0(\rho_*) - I) T_1^{\perp} =: D_*.
\end{equation}
Thus, $T_2$ is the matrix of the orthogonal projector onto $\Ran D_*^{\dagger}$. We summarize the arguments above in the following algorithm.

\begin{alg} \label{alg:T12}
Let the spectral data $\{ \la_{nk}, \al_{nk} \}_{(n, k) \in J}$ be given. We have to find $T_1$ and $T_2$.

\begin{enumerate}
    \item Find $T_1$ by~\eqref{relT1}.
    \item Construct the data $\{ \la_{nk}^0, \al_{nk}^0 \}_{(n, k) \in J}$, using~\eqref{rkAk} and~\eqref{sd0}.
    \item Construct $M^0(\la)$ by the formula
$$
M^0(\la) = \mathfrak M^0(\la) + C_*^0, \quad \mathfrak M^0(\la) = \sum_{(n, k) \in J} \left( \frac{1}{\la - \la_{nk}^0} + \be(\la_{nk}^0) \right) {\al_{nk}^0}',
$$
where $C_*^0$ can be found by~\eqref{C00}.
\item Find $E^0(\rho)$ by~\eqref{relE0}.
\item Fix $\rho_* \ne \rho_{nk}^0$, $n \in \mathbb Z$, $k = \overline{1, m}$, and construct the matrix $D_*$ by~\eqref{DT2}.
\item Determine $T_2$ as the matrix of the orthogonal projector onto $\Ran D_*^{\dagger}$ (e.g., by using Gram-Schmidt process).
\end{enumerate}
\end{alg}

\begin{proof}[Proof of Theorem~\ref{thm:uniq}]
Consider two problems $L$ and $\tilde L$ such that $\la_{nk} = \tilde \la_{nk}$, $\al_{nk} = \tilde \al_{nk}$, $(n, k) \in J$, $J = \tilde J$. The matrices $T_1$ and $T_2$ can be uniquely constructed by Algorithm~\ref{alg:T12}, so $T_1 = \tilde T_1$, $T_2 = \tilde T_2$.

Introduce the block matrix of spectral mappings $[P_{jk}(x, \la)]_{j, k = 1, 2}$ of size $(2 m \times 2 m)$ as follows:
$$
\begin{bmatrix}
 P_{11}(x, \la) & P_{12}(x, \la) \\
 P_{21}(x, \la) & P_{22}(x, \la) 
\end{bmatrix}
\begin{bmatrix}
 \tilde \vv(x, \la) & \tilde \Phi(x, \la) \\
 \tilde \vv^{[1]}(x, \la) & \tilde \Phi^{[1]}(x, \la)
\end{bmatrix}
=
\begin{bmatrix}
 \vv(x, \la) & \Phi(x, \la) \\
 \vv^{[1]}(x, \la) & \Phi^{[1]}(x, \la)
\end{bmatrix}.
$$
Recall that the matrix Wronskian $\langle (Y(x))^{\dagger}, Z(x) \rangle$ does not depend on $x$ if $Y$ and $Z$ are solutions of~\eqref{eqv}. Using this fact together with the definitions of $\vv(x, \la)$ and $\Phi(x, \la)$, it is easy to show that
$$
\begin{bmatrix}
 \tilde \vv(x, \la) & \tilde \Phi(x, \la) \\
 \tilde \vv^{[1]}(x, \la) & \tilde \Phi^{[1]}(x, \la)
\end{bmatrix}^{-1} =
\begin{bmatrix}
 (\tilde \Phi^{[1]}(x, \overline{\la}))^{\dagger} & -(\tilde \Phi(x, \overline{\la}))^{\dagger} \\
 -(\vv^{[1]}(x, \overline{\la}))^{\dagger} & (\tilde \vv(x, \overline{\la}))^{\dagger}.
\end{bmatrix}
$$
Consequently, we obtain
\begin{equation} \label{P12}
\left.
\begin{array}{ccccc}
P_{11} & = & \vv (\tilde \Phi^{[1]})^{\dagger} - \Phi (\tilde \vv^{[1]})^{\dagger} & = & I + (\vv - \tilde \vv) (\tilde \Phi^{[1]})^{\dagger} - (\Phi - \tilde \Phi) (\tilde \vv^{[1]})^\dagger \\
P_{12} & = & -\vv \tilde \Phi^{\dagger} + \Phi \tilde \vv^{\dagger} & = & -(\vv - \tilde \vv) \tilde \Phi^{\dagger} + (\Phi - \tilde \Phi) \tilde\vv^{\dagger}
\end{array}
\right\}
\end{equation}
where the appropriate arguments $(x, \la)$ and $(x, \overline{\la})$ are omitted for brevity. On the one hand, relation~\eqref{P12} and Lemma~\ref{lem:estsol} imply
\begin{equation} \label{asymptP}
P_{11}(x, \la) = I + o(1), \quad P_{12}(x, \la) = o(1), \quad |\rho| \to \iy, \quad \rho \in G_{\de},
\end{equation}
for each fixed $x \in [0, \pi]$. On the other hand, using \eqref{relPhi}, \eqref{P12}, and the relation $M(\la) = (M(\overline{\la}))^{\dagger}$, we derive
$$
P_{11} = \vv (\tilde \psi^{[1]})^{\dagger} - \psi (\tilde \vv^{[1]})^{\dagger} + \vv (\tilde M - M) (\tilde \vv^{[1]})^{\dagger}, \quad P_{12} = -\vv \tilde \psi^{\dagger} + \psi \tilde \vv^{\dagger} + \vv (M - \tilde M) \tilde \vv^{\dagger}.
$$
Lemma~\ref{lem:Weyl} says that the Weyl matrix $M(\la)$ can be recovered from the spectral data uniquely up to an additive constant. Hence, $(\tilde M(\la) - M(\la))$ is a constant matrix and the matrix functions $P_{11}(x, \la)$, $P_{12}(x, \la)$ are entire in $\la$ for each fixed $x \in [0, \pi]$. Therefore, asymptotics~\eqref{asymptP} together with Liouville's Theorem yield $P_{11}(x, \la) \equiv I$, $P_{12}(x, \la) \equiv 0$. Consequently, we have $\vv(x, \la) \equiv \tilde \vv(x, \la)$, $\Phi(x, \la) \equiv \tilde \Phi(x, \la)$.

Subtracting $\tilde \ell \tilde \vv = \la \tilde \vv$ from $\ell \vv = \la \vv$, we obtain
\begin{equation} \label{sm11}
((\sigma - \tilde \sigma) \vv)' = (\sigma - \tilde \sigma) \vv'
\end{equation}
a.e. on $(0, \pi)$. In addition, the matrix function $(\sigma - \tilde \sigma) \vv$ is absolutely continuous on $[0, \pi]$ for each fixed $\la$.
The same conclusions are valid for $\Phi$ instead of $\vv$. One can fix $\la$ and choose constant matrices $D_1$, $D_2$ so that 
$$
\det (\vv(x, \la) D_1 + \Phi(x, \la) D_2) \ne 0, \quad x \in [0, \pi].
$$
The matrix function 
$$
 (\sigma(x) - \tilde \sigma(x)) (\vv(x, \la) D_1 + \Phi(x, \la) D_2)
$$
is absolutely continuous with respect to $x \in [0, \pi]$. This implies that $(\sigma(x) - \tilde \sigma(x))$ is absolutely continuous on $[0, \pi]$. Using~\eqref{sm11}, we get $(\sigma - \tilde \sigma)' = 0$ a.e. on $(0, \pi)$. Thus, $\sigma(x) = \tilde \sigma(x) + H_1^{\diamond}$, where $H_1^{\diamond}$ is a Hermitian constant matrix. Using the initial conditions 
$$
\vv(0, \la) = \tilde \vv(0, \la) = T_1, \quad \vv^{[1]}(0, \la) = \tilde \vv^{[1]}(0, \la) = T_1^{\perp},
$$
we conclude that $(\sigma(0) - \tilde \sigma(0)) T_1 = 0$. This yields \eqref{H1perp}. The relation $V_2(\Phi) = \tilde V_2(\Phi)$ implies
$$
T_2 (\tilde H_2 - H_2 - H_1^{\diamond}) \Phi(\pi, \la) = 0.
$$
It follows from~\eqref{PhiPsi} that $\Phi(\pi, \la) = T_2 (V_1(\Psi))^{-1}$. Consequently, we obtain the relation for $H_2$ from~\eqref{transL}. This completes the proof.
\end{proof}

\section{Appendix}

In this section, we show how to represent Sturm-Liouville operators on graphs in the form~\eqref{eqv}-\eqref{bc2}. The Sturm-Liouville eigenvalue problem with singular potentials on the star-shaped graph with $m$ edges of equal length $\pi$ has the form
\begin{gather} \label{eqg}
    -(y_j^{[1]})' - \sigma_j(x_j) y_j^{[1]} - \sigma_j^2(x_j) y_j = \la y_j, \quad x \in (0, \pi), \quad j = \overline{1, m}, \\ \label{bcg}
    y_j(0) = 0, \quad j = \overline{1, m}, \\ \label{mc}
    y_1(\pi) = y_j(\pi), \quad j = \overline{2, m}, \qquad \sum_{j = 1}^m y_j^{[1]}(\pi) = h y_1(\pi),
\end{gather}
where $\{ \sigma_j \}_{j = 1}^m$ are real-valued functions from $L_2(0, \pi)$, $y_j^{[1]} := y_j' - \sigma_j y_j$, $y_j, y_j^{[1]} \in AC[0, \pi]$, $(y_j^{[1]})' \in L_2(0, \pi)$, $j = \overline{1, m}$, $h \in \mathbb R$. Conditions~\eqref{mc} generalize the standard matching conditions, which express Kirchoff's law in electrical circuits, balance of tension in elastic string network, etc. (see \cite{Kuch04, PPP04, BCFK06}).

The problem~\eqref{eqg}-\eqref{mc} is equivalent to \eqref{eqv}-\eqref{bc2} with $\sigma(x) = \diag\{ \sigma_j(x) \}_{j = 1}^m$ (the diagonal matrix with the diagonal entries $\{ \sigma_j(x) \}_{j = 1}^m$), $Y(x) = [y_j(x)]_{j = 1}^m$, $T_1 = 0$, $T_2 = [T_{2, jk}]_{j, k = 1}^m$, $T_{2, jk} = \tfrac{1}{m}$, $j, k = \overline{1, m}$, $H_2 = h T_2$. If there are the mixed boundary conditions instead of the Dirichlet boundary conditions~\eqref{bcg}:
$$
y_j^{[1]}(0) - h_j y_j(0) = 0, \quad h_j \in \mathbb R, \quad j = \overline{1, r}, \qquad
y_j(0) = 0, \quad j = \overline{r + 1, m},
$$
then $T_1 = [T_{1, jk}]_{j, k = 1}^m$, $T_{1, jk} = 1$ if $j = k \le r$ and $T_{1, jk} = 0$ otherwise.

Now consider an arbitrary geometrical graph $\mathcal G$ with the edges $\{ e_j \}_{j = 1}^m$ of equal length $\pi$. If the edge lengths are unequal but rationally dependent, then one can add auxiliary vertices to obtain a graph with equal edge lengths. For every edge $e_j$, $j = \overline{1, m}$, we introduce a parameter $x_j \in [0, \pi]$. Denote the ends of the edge $e_j$ by $w_{2j-1}$ and $w_{2j}$. The value $x_j = 0$ corresponds to the end $w_{2j-1}$ and $x_j = \pi$ corresponds to $w_{2j}$. Every vertex $v$ of the graph $\mathcal G$ is an equivalence class of the ends $w_j$ incident to this vertex:
$v = \{ w_{j_1}, w_{j_2}, \dots, w_{j_r} \}$. For $j = \overline{1, m}$, consider functions $y_j(x_j)$ and $\sigma_j(x_j)$, $x_j \in [0, \pi]$, from the classes described above. Denote
$$
    \begin{array}{ll}
    	y_{|w_{2j-1}} = y_j(0), \quad & y_{|w_{2j}} = y_j(\pi), \\
    	y^{[1]}_{|w_{2j-1}} = -y_j^{[1]}(0), \quad & y^{[1]}_{|w_{2j}} = y_j^{[1]}(\pi),
    \end{array}
    \quad j = \overline{1, m}.
$$	

Consider the Sturm-Liouville eigenvalue problem on the graph $\mathcal G$ given by equations~\eqref{eqg} on the edges and the following matching conditions in the vertices:
\begin{equation} \label{genmc}
\left.
\begin{array}{c}
y_{|w_j} = y_{|w_k}, \quad w_j, w_k \in v  \\
\sum\limits_{w_j \in v} y^{[1]}_{|w_j} = h_v y_{|w_{j_0}}, \quad w_{j_0} \in v
\end{array}
\right\} \quad v \in \mathcal V,
\end{equation}
where $\mathcal V$ is the set of the vertices of $\mathcal G$. Without loss of generality, we may assume that $\mathcal G$ is a bipartite graph, i.e., its vertices can be divided into two disjoint sets $\mathcal V_1$ and $\mathcal V_2$ so that each edge connects two vertices from different sets. To achieve this condition, one can add auxiliary vertices in the middle points of the edges. We may assume that all the vertices from $\mathcal V_1$ correspond to $x_j = 0$ and all the vertices from $\mathcal V_2$ correspond to $x_j = \pi$, i.e., if $w_{2j-1} \in v$, then $v \in \mathcal V_1$ and, if $w_{2j} \in v$, then $v \in \mathcal V_2$.

Fix a vertex $v \in \mathcal V_1$. Let $e_{j_1}$, $e_{j_2}$, \dots, $e_{j_r}$ be the edges incident to $v$. Construct the matrices $T_1^v = [T^v_{1, jk}]_{j, k = 1}^m$ and $H_1^v = h_v T_1^v$, $T_{1, jk}^v = \tfrac{1}{r}$ if $j, k \in \{ j_l \}_{l = 1}^r$ and $T_{1, jk}^v = 0$ otherwise. Put 
$$
T_1 := \sum_{v \in \mathcal V_1} T_1^v, \quad H_1 := \sum_{v \in \mathcal V_1} H_1^v.
$$
One can easily check that $T_1$ is an orthogonal projection matrix and $H_1 = H_1^{\dagger} = T_1 H_1 T_1$. The matrices $T_2$ and $H_2$ are constructed analogously by $\mathcal V_2$, $\sigma(x) := \diag\{ \sigma_j(x) \}_{j = 1}^m$. Then the Sturm-Liouville problem~\eqref{eqg},\eqref{genmc} on the graph $\mathcal G$ is equivalent to \eqref{eqv}-\eqref{bc2} with the constructed matrix coefficients.

Thus, the results of the paper are valid for Sturm-Liouville operators with singular potentials on arbitrary graphs having rationally dependent edge lengths. Matching conditions of other types than~\eqref{genmc} can be treated similarly (see, e.g., \cite{Kuch04, Now07}).

\medskip

\textbf{Acknowledgements.} This work was supported by Grant 19-71-00009 of the Russian Science Foundation.

\medskip

\noindent Natalia Pavlovna Bondarenko \\
1. Department of Applied Mathematics and Physics, Samara National Research University, \\
Moskovskoye Shosse 34, Samara 443086, Russia, \\
2. Department of Mechanics and Mathematics, Saratov State University, \\
Astrakhanskaya 83, Saratov 410012, Russia, \\
e-mail: {\it BondarenkoNP@info.sgu.ru}

\end{document}